\documentclass[10pt,a4paper]{article}

\usepackage[a4paper,margin=2.5cm]{geometry}
\usepackage{amsfonts,amsgen,amsmath,amstext,amsbsy,amsopn,amsthm,cases,listings}
\usepackage{amssymb,mathtools,mathrsfs,bm,bbm,dsfont}
\usepackage{microtype}
\usepackage{multirow,booktabs,longtable,float}
\usepackage{graphicx,epstopdf}
\usepackage{pgf,tikz,pgfplots}
\usetikzlibrary{arrows}
\pgfplotsset{compat=1.18}
\usepackage[marginal]{footmisc}
\usepackage{enumitem}
\usepackage{url}
\usepackage{authblk}
\usepackage{aliascnt}
\usepackage{wasysym,empheq,aligned-overset,subfigure}
\usepackage[T1]{fontenc}
\usepackage[thinc]{esdiff}
\usepackage[colorlinks=true,linkcolor=blue,citecolor=blue,urlcolor=blue]{hyperref}
\usepackage[nameinlink,capitalize]{cleveref}

\allowdisplaybreaks[1]
\setlist[itemize]{topsep=3pt,itemsep=2pt}
\setlist[enumerate]{topsep=3pt,itemsep=2pt}

\definecolor{uuuuuu}{rgb}{0.27,0.27,0.27}
\definecolor{sqsqsq}{rgb}{0.1255,0.1255,0.1255}

\theoremstyle{plain}
\newtheorem{theorem}{Theorem}[section]
\newaliascnt{lemma}{theorem}
\newtheorem{lemma}[lemma]{Lemma}
\aliascntresetthe{lemma}
\newaliascnt{proposition}{theorem}
\newtheorem{proposition}[proposition]{Proposition}
\aliascntresetthe{proposition}
\newaliascnt{corollary}{theorem}
\newtheorem{corollary}[corollary]{Corollary}
\aliascntresetthe{corollary}
\newaliascnt{claim}{theorem}

\aliascntresetthe{claim}

\theoremstyle{definition}
\newaliascnt{definition}{theorem}

\aliascntresetthe{definition}
\newaliascnt{remark}{theorem}

\aliascntresetthe{remark}
\newaliascnt{conjecture}{theorem}

\aliascntresetthe{conjecture}
\newaliascnt{problem}{theorem}

\aliascntresetthe{problem}
\newaliascnt{observation}{theorem}

\aliascntresetthe{observation}
\newaliascnt{example}{theorem}

\aliascntresetthe{example}
\theoremstyle{plain}

\crefname{theorem}{Theorem}{Theorems}
\crefname{proposition}{Proposition}{Propositions}
\crefname{lemma}{Lemma}{Lemmas}
\crefname{corollary}{Corollary}{Corollaries}
\crefname{claim}{Claim}{Claims}
\crefname{problem}{Problem}{Problems}
\crefname{definition}{Definition}{Definitions}
\crefname{remark}{Remark}{Remarks}
\crefname{section}{Section}{Sections}
\crefname{subsection}{Subsection}{Subsections}

\DeclareMathOperator{\ex}{ex}
\DeclareMathOperator{\supp}{supp}

\makeatletter
\newsavebox\myboxA
\newsavebox\myboxB
\newlength\mylenA
\newcommand*\xoverline[2][0.81]{%
    \sbox{\myboxA}{$\m@th#2$}%
    \setbox\myboxB\null%
    \ht\myboxB=\ht\myboxA%
    \dp\myboxB=\dp\myboxA%
    \wd\myboxB=#1\wd\myboxA%
    \sbox\myboxB{$\m@th\overline{\copy\myboxB}$}%
    \setlength\mylenA{\the\wd\myboxA}%
    \addtolength\mylenA{-\the\wd\myboxB}%
    \ifdim\wd\myboxB<\wd\myboxA%
       \rlap{\hskip 0.5\mylenA\usebox\myboxB}{\usebox\myboxA}%
    \else
        \hskip -0.5\mylenA\rlap{\usebox\myboxA}{\hskip 0.5\mylenA\usebox\myboxB}%
    \fi}
\makeatother

\newcommand{\calA}{\mathcal A}
\newcommand{\calB}{\mathcal B}

\newcommand{\calF}{\mathcal F}

\newcommand{\calT}{\mathcal T}

\newcommand{\HH}{\mathbb H}

\newcommand{\dirE}{\vec E}

\title{\bf\Large Upper Bounds on Tur\'an Densities via Extremal Set Theory}
\date{\today}
\author[1]{Yaobin~Chen\thanks{Email:\texttt{ybchen21@m.fudan.edu.cn}}}
\author[2]{Xizhi~Liu\thanks{Email:\texttt{liuxizhi@ustc.edu.cn}}}
\author[3]{Ningyuan~Yang\thanks{Email:\texttt{nyyang23@m.fudan.edu.cn}}}
\author[2]{Tianming~Zhu\thanks{Email:\texttt{zhutianming@mail.ustc.edu.cn}}}
\affil[1]{\small Shanghai Center for Mathematical Sciences, Fudan University, Shanghai, China}
\affil[2]{\small School of Mathematical Sciences, University of Science and Technology of China, Hefei, China}
\affil[3]{\small School of Mathematical Sciences, Fudan University, Shanghai, China}

\begin{document}
\maketitle
\begin{abstract}
    We exhibit, in a systematic way, connections between hypergraph Tur\'an problems and extremal set theory.  More specifically, we construct natural families of uniform hypergraphs for which the upper bounds on their Tur\'an densities reduce to classical problems in extremal set theory, including the Erd\H{o}s--Ko--Rado theorem, $L$-intersecting families, and the Erd\H{o}s matching problem.
\end{abstract}

% \renewcommand\thefootnote{}
% \footnotetext{\textit{Keywords:} hypergraph, Tur\'an density, entropy, homomorphism, extension construction.\\[0.15em]
% \textit{MSC2020:} 05C35, 05C65, 05D05.}
% \addtocounter{footnote}{-1}
% \renewcommand\thefootnote{\arabic{footnote}}

%%%%%%%%%%%%%%%%%%%%%%%%%%%%%%%%%%%%%%%%
\section{Introduction}\label{SEC:Introduction}

Let $r\ge2$ be an integer.  An \emph{$r$-uniform hypergraph}, or \emph{$r$-graph}, is a family of $r$-subsets of a vertex set.  For a hypergraph $H$, let $V(H)$ denote its vertex set, and write $v(H)\coloneqq |V(H)|$ for its number of vertices.  We identify a hypergraph with its edge set, and hence write $|H|$ for the number of edges of $H$.  

Given a family $\calF$ of $r$-graphs, an $r$-graph $H$ is \emph{$\calF$-free} if it contains no member of $\calF$ as a subgraph.  Its \emph{Tur\'an number} is
\[
        \ex(n,\calF)\coloneqq\max\left\{|H|: |V(H)|=n\text{ and } H\text{ is }\calF\text{-free}\right\},
\]
and its \emph{Tur\'an density} is
\[
        \pi(\calF)\coloneqq\lim_{n\to\infty}{\ex(n,\calF)}/{\tbinom{n}{r}}.
\]
We write $\ex(n,F)$ and $\pi(F)$ when $\calF=\{F\}$.

Tur\'an theory is one of the central topics in extremal combinatorics, beginning with Mantel's theorem \cite{Mantel07} for triangles and Tur\'an's theorem \cite{Turan41} for complete graphs.  For an ordinary graph $F$ with at least one edge, its Tur\'an density is governed by the chromatic number: the Erd\H{o}s--Stone--Simonovits theorem \cite{ErdosStone46,ErdosSimonovits66} gives $\pi(F)=1-1/(\chi(F)-1)$.  Hypergraph Tur\'an theory is far from well understood.  For $r\ge3$, there is no comparable chromatic-type principle, and even the value of $\pi(K_4^{3})$, one of the central problems already present in Tur\'an's work, remains open.  For broader context, see Sidorenko's survey \cite{Sidorenko95} and Keevash's survey \cite{Keevash11} on hypergraph Tur\'an theory.

Extremal set theory is another central area of extremal combinatorics.  Early landmarks include the Erd\H{o}s--Ko--Rado theorem \cite{EKR61} and Katona's foundational intersection theorem \cite{Katona64}, and the area asks how large a family of finite sets can be under restrictions on intersections, matchings, packings, shadows, and related forbidden configurations; representative results include Wilson's exact $t$-intersecting theorem \cite{Wilson84}, the Ahlswede--Khachatrian complete intersection theorem \cite{AhlswedeKhachatrian97}, the Frankl--F\"uredi forbidden-intersection theorems \cite{FranklFuredi85}, the Frankl--Wilson algebraic method \cite{FranklWilson81}, Bollob\'as's set-pair inequality \cite{Bollobas65}, R\"odl's packing theorem \cite{Rodl85}, and the Erd\H{o}s matching theorem \cite{ErdosMatching65}, while Frankl and Tokushige's survey \cite{FranklTokushige16} and monograph \cite{FranklTokushige18} give general references.

Connections between hypergraph Tur\'an problems and extremal set theory have appeared in several forms in the literature.  Lagrangian versions of intersecting-family and Erd\H{o}s--Ko--Rado type results can drive hypergraph Tur\'an theorems for extensions; examples include Hefetz and Keevash's work \cite{HefetzKeevash13}, the work of Jiang, Peng, and Wu on Lagrangian densities and Tur\'an numbers of extensions \cite{JiangPengWu18}, and Watts, Norin, and Yepremyan's work \cite{WattsNorinYepremyan19} in this direction.  Frankl's theorem \cite{Frankl90} states that if $k\ge1$ and $C^{2k}_3$ is the $2k$-uniform expanded triangle whose vertex set is the disjoint union $K_1\cup K_2\cup K_3$ of three $k$-sets and whose edges are $K_1\cup K_2$, $K_1\cup K_3$, and $K_2\cup K_3$, then $\pi(C^{2k}_3)=\frac12$.
His proof \cite{Frankl90} uses Katona's cycle method, a fundamental method in extremal set theory.

In this work, we initiate a more systematic study of connections between hypergraph Tur\'an problems and extremal set theory.  We first fix the extension convention used in the definitions below.  A \emph{partial hypergraph} is a finite vertex set together with a finite collection $e_1,\ldots,e_t$ of subsets (not necessarily of the same size) of its vertex set, called \emph{partial edges}.  If all partial edges have size at most $r$, its \emph{$r$-uniform extension} is the $r$-graph obtained by adding, for each $j\in[t]$, a private set $U_j$ of $r-|e_j|$ new vertices and declaring $e_j\cup U_j$ to be an edge.  The private sets are chosen disjointly for different partial edges; if $|e_j|=r$, then $U_j=\emptyset$.  Thus every specified partial edge is extended separately.
We write $[q]\coloneqq\{1,\ldots,q\}$ and $[p,q]\coloneqq\{p,p+1,\ldots,q\}$ for integer intervals, with the convention $[p,q]=\emptyset$ when $p>q$.  We now list the applications proved later in the paper.

%%%%%%%%%%%%%%%%%%%%%%%%%%%%%%%%%%%%%%%%
\subsection{Hypergraph triangles and \texorpdfstring{$L$}{L}-intersecting families}

Let $r,a,b,c,d$ be integers satisfying
\begin{equation}\label{eq:first-assumptions}
\begin{gathered}
        a\ge1,
        \qquad 0\le d\le c\le b\le r,
        \qquad b+c-d\le r,
        \quad\text{and}\quad a+b\le r.
\end{gathered}
\end{equation}
Define $T^r_{a,b,c,d}$ as follows (see \Cref{fig:triangle-construction}).  Choose disjoint sets $A$ and $E$ with $|A|=a$ and $|E|=r$, and choose $B,C\subseteq E$ such that $|B|=b$, $|C|=c$, and $|B\cap C|=d$.  Let $T^r_{a,b,c,d}$ be the $r$-uniform extension of the partial hypergraph on $A\cup E$ with partial edges $\{E,~ A\cup B,~ A\cup C\}$. 

\begin{figure}[ht]
\centering
\begin{tikzpicture}[scale=0.88, every node/.style={font=\small}]
    \coordinate (Ectr) at (0,0);
    \coordinate (Lctr) at (-1.08,0.78);
    \coordinate (Rctr) at (1.08,0.78);

    \fill[blue!18, opacity=0.50] (Ectr) ellipse (3.05 and 1.16);
    \begin{scope}
        \clip (Ectr) ellipse (3.05 and 1.16);
        \begin{scope}[rotate around={20:(Lctr)}]
            \fill[green!35, opacity=0.58] (Lctr) ellipse (2.18 and 1.34);
        \end{scope}
    \end{scope}
    \begin{scope}
        \clip (Ectr) ellipse (3.05 and 1.16);
        \begin{scope}[rotate around={-20:(Rctr)}]
            \fill[yellow!55!red!30, opacity=0.58] (Rctr) ellipse (2.18 and 1.34);
        \end{scope}
    \end{scope}
    \begin{scope}
        \clip[rotate around={20:(Lctr)}] (Lctr) ellipse (2.18 and 1.34);
        \clip[rotate around={-20:(Rctr)}] (Rctr) ellipse (2.18 and 1.34);
        \pgfseteorule
        \clip (-3.7,-1.5) rectangle (3.7,2.7)
            (Ectr) ellipse (3.05 and 1.16);
        \pgfsetnonzerorule
        \fill[red!28, opacity=0.70] (-3.7,-1.5) rectangle (3.7,2.7);
    \end{scope}

    \draw[very thick, black] (Ectr) ellipse (3.05 and 1.16);
    \begin{scope}[rotate around={20:(Lctr)}]
        \draw[very thick, black] (Lctr) ellipse (2.18 and 1.34);
    \end{scope}
    \begin{scope}[rotate around={-20:(Rctr)}]
        \draw[very thick, black] (Rctr) ellipse (2.18 and 1.34);
    \end{scope}

    \node at (0,1.61) {$A$};
    \node at (-1.24,0.12) {$B\setminus C$};
    \node at (1.24,0.12) {$C\setminus B$};
    \node at (0,0.48) {$B\cap C$};
    \node at (-2.35,1.12) {$U_B$};
    \node at (2.35,1.12) {$U_C$};
    \node at (0,-1.48) {$E$};
    \node[anchor=east] at (-2.38,1.72) {$A\cup B\cup U_B$};
    \node[anchor=west] at (2.38,1.72) {$A\cup C\cup U_C$};
\end{tikzpicture}
\caption{The $r$-graph $T^r_{a,b,c,d}$.  The three ellipses are the three $r$-edges after extension; $U_B$ and $U_C$ are the private vertices added to the two side partial edges.}
\label{fig:triangle-construction}
\end{figure}
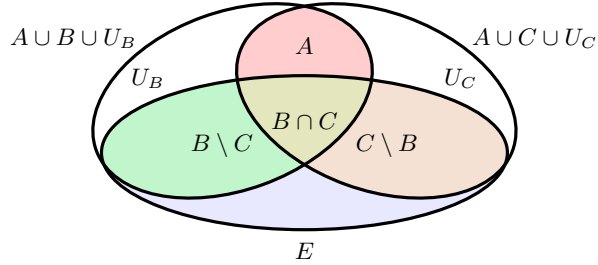

For $L\subseteq[0,b]$, a family $\calB\subseteq\tbinom{[r]}{b}$ is \emph{$L$-intersecting} if $|B\cap B'|\in L$ for all distinct $B,B'\in\calB$.  Define
\[
        \Psi(r,b,L)\coloneqq\max\left\{|\calB|:\calB\subseteq\tbinom{[r]}{b}
        \text{ is }L\text{-intersecting}\right\}.
\]

The first result reduces the upper bound for this extension triangle to the corresponding $L$-intersecting problem on $b$-subsets of $[r]$.

\begin{theorem}\label{thm:first-main}
Assume \cref{eq:first-assumptions}, and put $L \coloneqq [0,d-1]\cup[b-c+d+1,b]$.  Then
\[
        \pi\bigl(T^r_{a,b,c,d}\bigr)
        \le \frac{\Psi\left(r,b,L\right)}{\binom{r}{b}}.
\]
In particular, if $i\ge1$ and $r \ge 2i$, then $\pi(T^r_{i,r-i,r-i,r-2i})\le\frac{i}{r}$. 
\end{theorem}

It is useful to record the following complement symmetry.  For $L\subseteq[0,b]$, put
\[
        L^{\ast}_{r,b}\coloneqq
        \left\{r-2b+\ell:\ell\in L\right\}\cap[0,r-b].
\]
Then
\begin{equation}\label{eq:Psi-complement}
        \Psi(r,b,L)=\Psi(r,r-b,L^{\ast}_{r,b}).
\end{equation}

Indeed, the map $B\mapsto [r]\setminus B$ is a bijection from $\tbinom{[r]}{b}$ to
$\tbinom{[r]}{r-b}$.  For distinct $B,B'\in\tbinom{[r]}{b}$,
\[
        |([r]\setminus B)\cap([r]\setminus B')|
        =r-2b+|B\cap B'|.
\]
Thus the bijection sends $L$-intersecting families of $b$-sets exactly to
$L^{\ast}_{r,b}$-intersecting families of $(r-b)$-sets.

This explains the special case in \cref{thm:first-main}.  If $i\ge1$, $2i\le r$, and $(a,b,c,d) = (i, r-i, r-i, r-2i)$, then $L=[0,r-2i-1]\cup[r-2i+1,r-i]$, and \cref{eq:Psi-complement} gives $\Psi(r,r-i,L)=\Psi(r,i,[1,i])$.  The Erd\H{o}s--Ko--Rado theorem \cite{EKR61} gives $\Psi(r,i,[1,i])=\binom{r-1}{i-1}$ for $2i\le r$, so the ratio in \cref{thm:first-main} is ${\binom{r-1}{i-1}}/{\binom{r}{i}}=\frac{i}{r}$.

More generally, for every prescribed set $L$ of allowed intersection sizes, one can choose a forbidden family whose Tur\'an-density upper bound is governed by the corresponding $L$-intersecting problem. 

\begin{theorem}\label{thm:arbitrary-L}
Let $1\le b<r$, let $a\ge1$ satisfy $a+b\le r$, and let $L\subseteq[0,b]$.  Put
\[
        J_{r,b}(L)\coloneqq
        \left\{j\in[0,b-1]\setminus L: 2b-j\le r\right\}
        \quad\text{and}\quad 
        \calT^r_{a,b,L}\coloneqq
        \left\{T^r_{a,b,b,j}:j\in J_{r,b}(L)\right\}.
\]
Then
\[
        \pi(\calT^r_{a,b,L})
        \le \frac{\Psi(r,b,L)}{\binom{r}{b}}.
\]
\end{theorem}

Consequently, classical estimates for $L$-intersecting families can be inserted directly into \cref{thm:arbitrary-L}.  For instance, the case $L=[t,b]$ is the usual $t$-intersecting problem, treated by the Erd\H{o}s--Ko--Rado theorem \cite{EKR61}, Wilson's exact theorem \cite{Wilson84}, and the Ahlswede--Khachatrian complete intersection theorem \cite{AhlswedeKhachatrian97}, among others.  Forbidden-intersection sets of Frankl--F\"uredi type \cite{FranklFuredi85}, such as $L(\alpha,\beta)=[0,\alpha-1]\cup[b-\beta,b]$, give further explicit bounds; more general algebraic bounds for prescribed intersection sets go back to the Frankl--Wilson method \cite{FranklWilson81} as well.  Frankl and Tokushige's survey \cite{FranklTokushige16} and monograph \cite{FranklTokushige18} give a broad account of these intersection theorems.

%%%%%%%%%%%%%%%%%%%%%%%%%%%%%%%%%%%%%%%%
\subsection{Matching constructions and the Erd\H{o}s matching problem}

Let $s\ge2$ and let $r,k,a$ be integers satisfying
\begin{equation}\label{eq:matching-assumptions}
\begin{gathered}
        1\le a\le k<r
        \quad\text{and}\quad sk\le r.
\end{gathered}
\end{equation}
Define $T^r_{a,[k,s]}$ as follows.  Choose disjoint sets $A$ and $E$ with $|A|=a$ and $|E|=r$, and choose pairwise disjoint sets $B_1,\ldots,B_s\subseteq E$ with $|B_i|=k$ for every $i\in[s]$.  Put $\bar B_i\coloneqq E\setminus B_i$.  Let $T^r_{a,[k,s]}$ be the $r$-uniform extension of the partial hypergraph on $A\cup E$ with partial edges
\[
        \{E,~ A\cup\bar B_1,~\ldots,~ A\cup\bar B_s\}.
\]
Let
\[
        M(r,k,s-1)\coloneqq\max\left\{|\calB|:\calB\subseteq\tbinom{[r]}{k}\text{ and } \nu(\calB)\le s-1\right\},
\]
where $\nu(\calB)$ is the \emph{matching number} of $\calB$.

\begin{theorem}\label{thm:matching-main}
Under the assumptions in \cref{eq:matching-assumptions},
\[
        \pi\bigl(T^r_{a,[k,s]}\bigr)
        \le \frac{M(r,k,s-1)}{\binom{r}{k}}.
\]
\end{theorem}

The extremal function $M(r,k,s-1)$ is the quantity in the Erd\H{o}s matching conjecture \cite{ErdosMatching65}.  In the present notation, the conjecture predicts
\[
        M(r,k,s-1)
        =
        \max\left\{\tbinom{ks-1}{k},
        \tbinom{r}{k}-\tbinom{r-s+1}{k}\right\};
\]
the two terms come, respectively, from all $k$-sets inside a fixed $(ks-1)$-set and from all $k$-sets meeting a fixed $(s-1)$-set.  The large-ground-set case needed here goes back to Erd\H{o}s \cite{ErdosMatching65}, with important subsequent refinements by Bollob\'as--Daykin--Erd\H{o}s \cite{BollobasDaykinErdos76}, Huang--Loh--Sudakov \cite{HuangLohSudakov12}, Frankl \cite{Frankl13Matching}, and {\L}uczak--Mieczkowska \cite{LuczakMieczkowska14}, while Frankl and Tokushige's monograph \cite{FranklTokushige18} gives a general account of matching problems in extremal set theory.

Thus, when $k$ and $s$ are fixed and $r$ is sufficiently large, the Erd\H{o}s matching theorem evaluates
\[
        M(r,k,s-1)=\tbinom{r}{k}-\tbinom{r-s+1}{k},
\]
so \cref{thm:matching-main} gives
\[
        \pi\bigl(T^r_{a,[k,s]}\bigr)
        \le
        1-\frac{\binom{r-s+1}{k}}{\binom{r}{k}}.
\]

%%%%%%%%%%%%%%%%%%%%%%%%%%%%%%%%%%%%%%%%
\subsection{Product constructions and products of \texorpdfstring{$L$}{L}-intersecting families}

Let $m\ge1$ and $r\ge2$.  Let $a\ge0$, and for each $i\in[m]$ let
\begin{equation}\label{eq:product-coordinate-assumptions}
\begin{gathered}
        0\le d_i\le c_i\le b_i\le r
        \quad\text{and}\quad b_i+c_i-d_i\le r.
\end{gathered}
\end{equation}
Assume
\begin{equation}\label{eq:product-size-assumptions}
\begin{gathered}
        q_0\coloneqq a+\sum_{i\in [m]} b_i\le r
        \quad\text{and}\quad
        a+\sum_{i\in [m]} c_i\le r.
\end{gathered}
\end{equation}
Define $T^r_{a,b_1,c_1,d_1,\ldots,b_m,c_m,d_m}$ as follows.  Choose pairwise disjoint $r$-sets $E_1,\ldots,E_m$, choose a set $A$ disjoint from their union with $|A|=a$, and choose $B_i,C_i\subseteq E_i$ satisfying $|B_i|=b_i$, $|C_i|=c_i$, and $|B_i\cap C_i|=d_i$ for every $i\in[m]$.  Let $T^r_{a,b_1,c_1,d_1,\ldots,b_m,c_m,d_m}$ be the $r$-uniform extension of the partial hypergraph on $A\cup\bigcup_{i\in [m]} E_i$ with partial edges
\[
        \left\{E_1,\ldots,E_m,~
        A\cup\bigcup_{i\in [m]} B_i,~
        A\cup\bigcup_{i\in [m]} C_i\right\}.
\]

\begin{theorem}\label{thm:product-main}
Under the assumptions in \cref{eq:product-coordinate-assumptions,eq:product-size-assumptions},
\[
        \pi\bigl(T^r_{a,b_1,c_1,d_1,\ldots,b_m,c_m,d_m}\bigr)
        \le
        \max_{i\in[m]}\frac{\Psi(r,b_i,L_i)}{\binom{r}{b_i}}, 
\]
where $L_i\coloneqq[0,d_i-1]\cup[b_i-c_i+d_i+1,b_i]$ for $i \in [m]$. 
\end{theorem}

Note that the case $m=1$ already contains the result in \cref{thm:first-main}.  We keep the separate proof in \cref{thm:first-main} because it is more transparent and motivates the later reductions.

The paper is organized as follows.  In \cref{SEC:Preliminaries}, we recall the Chao--Yu entropy reduction and the homomorphism monotonicity used throughout.  In \cref{sec:proofs}, we apply this reduction directly to prove the three main bounds.  In \cref{sec:hom-reductions}, we refine the bounds of \cref{thm:first-main,thm:product-main} by applying homomorphic reductions to the forbidden extension constructions.  We end with concluding remarks.

%%%%%%%%%%%%%%%%%%%%%%%%%%%%%%%%%%%%%%%%
\section{Preliminaries}\label{SEC:Preliminaries}

A \emph{homomorphism} $F\to H$ between $r$-graphs is a map $\phi:V(F)\to V(H)$ such that $\phi(e)$ is an edge of $H$ for every $e\in F$; in particular, $\phi$ is injective on every edge of $F$.  An $r$-graph $H$ is \emph{$\calF$-hom-free} if no $F\in\calF$ admits a homomorphism to $H$.
We use $F\to H$ to mean that there exists a homomorphism from $F$ to $H$. 

A \emph{homomorphism from a partial hypergraph} to an $r$-graph $H$ is a map that is injective on every partial edge and sends every partial edge into some edge of $H$.

For a finite set $S$, we write $V(H)^S$ for the set of maps from $S$ to $V(H)$.

Recall that a \emph{simplicial complex} is a family of finite sets closed under taking subsets; its members are called \emph{faces}. 
We call a simplicial complex a \emph{maximal partial hypergraph}; if all faces have size at most $r$, it is a \emph{maximal partial $r$-graph}.\footnote{Chao and Yu \cite{CY24} call this object a partial hypergraph.  In the present paper, the phrase \emph{partial hypergraph} is reserved for the objects used in the extension constructions, where every specified partial edge is extended separately.  We use \emph{maximal partial hypergraph} for the Chao--Yu object in order to keep the two conventions distinct.}  The word ``maximal'' refers to specifying the complex by its inclusion-maximal faces.  A homomorphism from a maximal partial $r$-graph to an $r$-graph $H$ is understood in the same sense, applied to every face.

\begin{lemma}\label{lem:extension-equivalence}
A partial hypergraph whose partial edges have size at most $r$ admits a homomorphism to an $r$-graph $H$ if and only if its $r$-uniform extension admits a homomorphism to $H$.
\end{lemma}

\begin{proof}
A homomorphism from the extension restricts to a homomorphism from the partial hypergraph.  Conversely, suppose $\phi$ maps every partial edge $e$ injectively into some edge $E_e$ of $H$.  Since $|E_e|=r$, the remaining vertices of $E_e$ can be assigned to the private vertices $U_e$.  These private sets are disjoint and appear in no other partial edge, so the choices are independent and extend $\phi$ to a homomorphism of the extension.
\end{proof}

We will use the following standard homomorphism monotonicity for Tur\'an density; see, for example, Keevash's survey \cite{Keevash11}.

\begin{proposition}\label{prop:hom-monotonicity}
    Let $F$ and $F'$ be two $r$-graphs. Suppose that there exists a homomorphism from $F$ to $F'$. Then $\pi(F)\le\pi(F')$. More generally, if for every $F'\in\calF'$ there exists $F\in\calF$ with $F\to F'$, then $\pi(\calF)\le\pi(\calF')$.
\end{proposition}

Following Chao and Yu \cite{CY24}, we recall the ratio-sequence framework.  A \emph{random ordered edge} of an $r$-graph $H$ is a random variable supported on
\[
        \dirE(H)\coloneqq\left\{(v_1,\ldots,v_r):\{v_1,\ldots,v_r\}\in H\right\}.
\]
It is \emph{symmetric} if its law is invariant under every permutation of the coordinates.  For a symmetric random ordered edge $(X_1,\ldots,X_r)$, define
\[
        x_i\coloneqq2^{\HH(X_i\mid X_{i+1},\ldots,X_r)-\HH(X_i)}\quad\text{for}\quad i\in[r].
\]
Then $0<x_1\le\cdots\le x_r=1$; this follows from symmetry and monotonicity of conditional entropy.  We use the convention that conditioning on an empty tuple is omitted, so $x_r=1$.  If $\beta\coloneqq\prod_{i\in [r]} x_i$, then the chain rule and symmetry give the entropy identity
\[
        \HH(X_1,\ldots,X_r)
        =\sum_{i\in [r]} \HH(X_i\mid X_{i+1},\ldots,X_r)
        =r\HH(X_1)+\log_2\beta.
\]
This identity will be used repeatedly below.

The hom-free entropy reduction of Chao and Yu \cite[Corollary~5.6]{CY24} gives the following form.  If $\calF$ is a finite family of $r$-graphs, then
\begin{equation}\label{eq:hom-free-entropy-reduction}
        \pi(\calF)\le
        \sup\left\{
        \prod_{i\in [r]} x_i:
        \begin{array}{l}
        H\text{ is a finite }\calF\text{-hom-free }r\text{-graph with at least one edge,}\\
        (X_1,\ldots,X_r)\text{ is a symmetric random ordered edge of }H
        \end{array}
        \right\}.
\end{equation}

\begin{lemma}[{\cite[Lemma~3.9]{CY24}}]\label{lem:mixture}
Let $Y_1,\ldots,Y_N$ be random variables taking values in a common finite set $\Omega$.  Suppose every point of $\Omega$ belongs to the support of at most $q$ of the variables $Y_i$.  Then there is a \emph{mixture} $Z$ of $Y_1,\ldots,Y_N$ such that
\[
        \sum_{i\in [N]}2^{\HH(Y_i)}\le q\,2^{\HH(Z)}.
\]
\end{lemma}
This is \cite[Lemma~3.9]{CY24} with their parameter $a=q$; their support condition means exactly that each point is contained in at most $q$ supports.

A partial hypergraph whose partial edges have size at most $r$ generates a \emph{maximal partial $r$-graph}, namely the simplicial complex generated by its partial edges.  If $P$ is a maximal partial $r$-graph and $<$ is a linear order on $V(P)$, then for $v\in V(P)$ let $M_{P,<}(v)$ be the family of faces $e$ such that $v$ is the maximum vertex of $e$ under $<$.  We say $P$ is a \emph{partial forest with respect to $<$} if for every $v$, the family $M_{P,<}(v)$ has a unique inclusion-maximal member.  If this maximal member has size $j$, then $v$ contributes one to $n_j$.  The vector $(n_1,\ldots,n_r)$ is the \emph{forest sequence} of $(P,<)$.

\begin{lemma}[{\cite[Lemma~6.5]{CY24}}]\label{lem:forest-sampling}
Let $(X_1,\ldots,X_r)$ be a symmetric random ordered edge of an $r$-graph $H$, with ratio sequence $x_1,\ldots,x_r$.  Let $P$ be a partial forest with respect to a linear order $<$, and let $(n_1,\ldots,n_r)$ be its forest sequence.  Then there is a random homomorphism $(Y_v)_{v\in V(P)}$ from $P$ to $H$ such that
\[
        \HH((Y_v)_{v\in V(P)})
        =|V(P)|\HH(X_1)+\log_2\left(\prod_{i\in [r]} x_i^{n_{r+1-i}}\right).
\]
Moreover, for every face $e$ of $P$ of size $j$, the joint distribution of $(Y_v)_{v\in e}$ is the same as that of $(X_{r-j+1},\ldots,X_r)$ up to a permutation of coordinates.
\end{lemma}
This is \cite[Lemma~6.5]{CY24}, with $k=r$ and the partial-forest terminology from Chao--Yu's Definition~6.3 in \cite{CY24}.

%%%%%%%%%%%%%%%%%%%%%%%%%%%%%%%%%%%%%%%%
\section{Proofs}\label{SEC:Proof}\label{sec:proofs}

In this section, we present the proofs of the three bounds.  For each  family, the main step is a block-product estimate for the ratio sequence; the corresponding Tur\'an bound then follows from the entropy reduction recalled above.

\subsection{Proof of \texorpdfstring{\cref{thm:first-main,thm:arbitrary-L}}{Theorems 1.1 and 1.2}}

Throughout this subsection, let $F\coloneqq T^r_{a,b,c,d}$ and $L\coloneqq[0,d-1]\cup[b-c+d+1,b]$.

\begin{proposition}\label{prop:first-block}
Let $H$ be a finite $F$-hom-free $r$-graph with at least one edge, and let $(X_1,\ldots,X_r)$ be any symmetric random ordered edge of $H$ with ratio sequence $0<x_1\le\cdots\le x_r=1$.  Then
\[
        \prod_{j=r-b-a+1}^{r-b}x_j
        \le \frac{\Psi(r,b,L)}{\binom{r}{b}}.
\]
Consequently, by \cref{eq:hom-free-entropy-reduction}, $\pi(F)\le {\Psi(r,b,L)}/{\binom{r}{b}}$. 
\end{proposition}

\begin{proof}
Let $\beta\coloneqq\prod_{j\in [r]} x_j$.  Fix an $r$-set $E=\{v_1,\ldots,v_r\}$ and an $a$-set $A=\{w_1,\ldots,w_a\}$ disjoint from $E$.  For every $B\in\tbinom{[r]}{b}$, write $\hat{B}\coloneqq\left\{v_i:i\in B\right\}$ and let $P_B$ be the maximal partial $r$-graph generated by the faces $\{E, A\cup\hat{B}\}$. 
Order the vertices by $v_1<\cdots<v_r<w_1<\cdots<w_a$.  Then $P_B$ is a partial forest: the vertex $v_i$ has the prefix $\{v_1,\ldots,v_i\}$ as its unique maximal face with maximum $v_i$, and the vertex $w_\ell$ has $\hat{B}\cup\{w_1,\ldots,w_\ell\}$ as its unique maximal face with maximum $w_\ell$.  Hence the forest sequence is
\[
        n_j\coloneqq\begin{cases}2,&b+1\le j\le b+a,\\1,&\text{otherwise.}\end{cases}
\]
By \cref{lem:forest-sampling}, there is a random homomorphism $Y^B$ from $P_B$ to $H$ such that
\[
        \HH(Y^B)=(r+a)\HH(X_1)+\log_2(\beta\Gamma),
        \quad\text{where}\quad 
        \Gamma\coloneqq\prod_{j=r-b-a+1}^{r-b}x_j.
\]
This entropy is independent of $B$.

For $\phi\in V(H)^{E\cup A}$ define
\[
        \calB(\phi)\coloneqq\left\{B\in\tbinom{[r]}{b}:\phi\in\supp(Y^B)\right\}.
\]
We claim that $\calB(\phi)$ is $L$-intersecting.  Suppose not.  Then there are distinct $B_1,B_2\in\calB(\phi)$ with
\[
        d\le |B_1\cap B_2|\le b-c+d.
\]
Choose a $c$-set $C\subseteq B_2$ such that $|B_1\cap C|=d$: choose $d$ points from $B_1\cap B_2$ and $c-d$ points from $B_2\setminus B_1$.  This is possible by the displayed inequalities.  Let $\hat{C}\coloneqq\left\{v_i:i\in C\right\}$.  Since $C\subseteq B_2$, the face $A\cup\hat{C}$ belongs to $P_{B_2}$.  Hence $P_{B_1}\cup P_{B_2}$ contains the partial hypergraph defining $T^r_{a,b,c,d}$, with partial edges $E$, $A\cup\hat{B}_1$, and $A\cup\hat{C}$.  Since $Y^{B_1}$ and $Y^{B_2}$ are supported on homomorphisms from $P_{B_1}$ and $P_{B_2}$ to $H$, respectively, and $\phi$ lies in both supports, these three partial edges are mapped injectively into edges of $H$.  Thus $\phi$ gives a homomorphism from this partial hypergraph to $H$, contradicting $F$-hom-freeness by \cref{lem:extension-equivalence}.  Thus $|\calB(\phi)|\le\Psi(r,b,L)$ for all $\phi$.

By \cref{lem:mixture}, there is a mixture $Z$ of the $Y^B$ such that
\[
        \tbinom{r}{b} 2^{\HH(Y^B)}\le\Psi(r,b,L)2^{\HH(Z)}.
\]
For every $B$, the marginal of $Y^B$ on $E$ has the same law as $(X_1,\ldots,X_r)$, and every one-vertex marginal on $A$ has the same law as $X_1$.  The same holds for the mixture $Z$.  Therefore
\[
        \HH(Z)\le \HH(X_1,\ldots,X_r)+a\HH(X_1)
        =(r+a)\HH(X_1)+\log_2\beta.
\]
Substitution gives
\[
        \tbinom{r}{b} 2^{(r+a)\HH(X_1)}\beta\Gamma
        \le
        \Psi(r,b,L)2^{(r+a)\HH(X_1)}\beta,
\]
and cancellation yields the block-product bound.  Since all $x_j\le1$, the full product $\prod_{j\in [r]} x_j$ is at most this block product.  By \cref{eq:hom-free-entropy-reduction}, this gives the asserted Tur\'an bound.
\end{proof}

\begin{proof}[Proof of \cref{thm:arbitrary-L}]
Let $\calF\coloneqq\calT^r_{a,b,L}$, and let $H$ be any finite $\calF$-hom-free
$r$-graph with at least one edge.  Let $(X_1,\ldots,X_r)$ be any symmetric random ordered edge of $H$
with ratio sequence $0<x_1\le\cdots\le x_r=1$, and put
$\beta\coloneqq\prod_{j\in [r]} x_j$.

Fix an $r$-set $E=\{v_1,\ldots,v_r\}$ and an $a$-set
$A=\{w_1,\ldots,w_a\}$ disjoint from $E$.  For every
$B\in\tbinom{[r]}{b}$, define $P_B$ and $Y^B$ exactly as in the proof of
\cref{prop:first-block}.  Thus
\[
        \HH(Y^B)=(r+a)\HH(X_1)+\log_2(\beta\Gamma),
        \quad\text{where}\quad
        \Gamma\coloneqq\prod_{j=r-b-a+1}^{r-b}x_j.
\]
For $\phi\in V(H)^{E\cup A}$ define
\[
        \calB(\phi)\coloneqq\left\{B\in\tbinom{[r]}{b}:\phi\in\supp(Y^B)\right\}.
\]
We claim that $\calB(\phi)$ is $L$-intersecting.  If not, take distinct
$B_1,B_2\in\calB(\phi)$ with $j\coloneqq |B_1\cap B_2|\notin L$.  Since
$B_1$ and $B_2$ are distinct $b$-subsets of $[r]$, one has $j\le b-1$ and
$2b-j\le r$, so $j\in J_{r,b}(L)$.  Moreover
$P_{B_1}\cup P_{B_2}$ contains the partial hypergraph defining
$T^r_{a,b,b,j}$, with partial edges $E$, $A\cup\hat{B}_1$, and
$A\cup\hat{B}_2$.  Since $\phi$ lies in both supports, these partial
edges are mapped injectively into edges of $H$, giving a homomorphism from
$T^r_{a,b,b,j}$ to $H$ by \cref{lem:extension-equivalence}.  This contradicts
$\calF$-hom-freeness.  Hence $|\calB(\phi)|\le\Psi(r,b,L)$ for all $\phi$.

\Cref{lem:mixture} gives a mixture $Z$ of the variables $Y^B$ such that
\[
        \tbinom{r}{b}2^{\HH(Y^B)}
        \le \Psi(r,b,L)2^{\HH(Z)}.
\]
As in the proof of \cref{prop:first-block},
\[
        \HH(Z)\le (r+a)\HH(X_1)+\log_2\beta.
\]
Substitution and cancellation give
\[
        \prod_{j=r-b-a+1}^{r-b}x_j
        \le \frac{\Psi(r,b,L)}{\binom{r}{b}}.
\]
Since all $x_j\le1$, the full product $\prod_{j\in [r]} x_j$ is at most this
block product.  By \cref{eq:hom-free-entropy-reduction}, this gives the desired
bound.
\end{proof}

\subsection{Proof of \texorpdfstring{\cref{thm:matching-main}}{Theorem 1.3}}

The argument is parallel to the proof of \cref{thm:first-main}, with the role of an intersecting-family bound replaced by a matching-number bound.  The forbidden graph forces the family of admissible $k$-sets arising from any fixed support point to have matching number at most $s-1$.  Let $F\coloneqq T^r_{a,[k,s]}$.

\begin{proposition}\label{prop:matching-block}
Let $H$ be a finite $F$-hom-free $r$-graph with at least one edge, and let $(X_1,\ldots,X_r)$ be any symmetric random ordered edge of $H$ with ratio sequence $0<x_1\le\cdots\le x_r=1$.  Then
\[
        \prod_{j=k-a+1}^{k}x_j
        \le\frac{M(r,k,s-1)}{\binom{r}{k}}.
\]
Consequently, by \cref{eq:hom-free-entropy-reduction}, $\pi(F)\le {M(r,k,s-1)}/{\binom{r}{k}}$. 
\end{proposition}

\begin{proof}
Let $\beta\coloneqq\prod_{j\in [r]} x_j$.  Fix an $r$-set $E=\{v_1,\ldots,v_r\}$ and an $a$-set $A=\{w_1,\ldots,w_a\}$ disjoint from $E$.  For each $D\in\tbinom{[r]}{k}$, write $\hat{D}\coloneqq\left\{v_i:i\in D\right\}$.  Let $P_D$ be the maximal partial $r$-graph generated by $\{E, A\cup(E\setminus\hat{D})\}$. 
Order the vertices by $v_1<\cdots<v_r<w_1<\cdots<w_a$.  As in \cref{prop:first-block}, this is a partial forest with one contribution in every size $1,\ldots,r$ and one additional contribution in the sizes $r-k+1,\ldots,r-k+a$.  Thus \cref{lem:forest-sampling} gives a random homomorphism $Y^D$ from $P_D$ to $H$ satisfying
\[
        \HH(Y^D)=(r+a)\HH(X_1)+\log_2(\beta\Gamma), 
        \quad\text{where}\quad
        \Gamma\coloneqq\prod_{j=k-a+1}^{k}x_j.
\]
For $\phi\in V(H)^{E\cup A}$ define
\[
        \calB(\phi)\coloneqq\left\{D\in\tbinom{[r]}{k}:\phi\in\supp(Y^D)\right\}.
\]
If $\calB(\phi)$ contained $s$ pairwise disjoint sets $D_1,\ldots,D_s$, then taking $B_i=\hat{D}_i$ gives $\bar B_i=E\setminus\hat{D}_i$, and $P_{D_1}\cup\cdots\cup P_{D_s}$ would contain the partial hypergraph defining $T^r_{a,[k,s]}$, with partial edges $E,A\cup(E\setminus\hat{D}_1),\ldots,A\cup(E\setminus\hat{D}_s)$.  Since $\phi$ lies in each support, these partial edges are mapped injectively into edges of $H$, giving a homomorphism from the partial hypergraph to $H$ and hence, by \cref{lem:extension-equivalence}, a homomorphism from $T^r_{a,[k,s]}$ to $H$.  This contradicts $F$-hom-freeness.  Therefore $\nu(\calB(\phi))\le s-1$, and
\[
        |\calB(\phi)|\le M(r,k,s-1)
        \qquad\text{for all }\phi.
\]

\Cref{lem:mixture} now gives a mixture $Z$ of the variables $Y^D$ such that
\[
        \tbinom{r}{k} 2^{\HH(Y^D)}\le M(r,k,s-1)2^{\HH(Z)}.
\]
The marginal of $Z$ on $E$ is the law of $(X_1,\ldots,X_r)$, and each vertex of $A$ has marginal $X_1$.  Thus, by subadditivity and the entropy identity above,
\[
        \HH(Z)\le \HH(X_1,\ldots,X_r)+a\HH(X_1)
        =(r+a)\HH(X_1)+\log_2\beta.
\]
Substituting the formulas for $\HH(Y^D)$ and $\HH(Z)$ gives
\[
        \tbinom{r}{k} 2^{(r+a)\HH(X_1)}\beta\Gamma
        \le
        M(r,k,s-1)2^{(r+a)\HH(X_1)}\beta.
\]
After cancellation, the block-product bound follows.  Since all $x_j\le1$, the full product $\prod_{j\in [r]} x_j$ is at most this block product.  By \cref{eq:hom-free-entropy-reduction}, this gives the asserted Tur\'an bound.
\end{proof}

\subsection{Proof of \texorpdfstring{\cref{thm:product-main}}{Theorem 1.4}}

Let $V_1,\ldots,V_m$ be pairwise disjoint sets, each of size $n$.  For $0\le r_i\le n$ and $L_i\subseteq[0,r_i]$, recall that 
\[
        \Psi(n,r_i,L_i)\coloneqq\max\left\{|\calA|:\calA\subseteq\tbinom{V_i}{r_i}\text{ and }
        \ |A\cap A'|\in L_i\text{ for all distinct }A,A'\in\calA\right\}.
\]

We first record the product set-theoretic estimate that will be used to control the support families in the proof of the product bound.

\begin{theorem}\label{thm:product-intersection}
Let $\calA\subseteq\tbinom{V_1}{r_1}\times\cdots\times\tbinom{V_m}{r_m}$ have the following property: for every two distinct tuples $A=(A_1,\ldots,A_m)$ and $B=(B_1,\ldots,B_m)$ in $\calA$, there exists $i\in[m]$ such that $|A_i\cap B_i|\in L_i$.  Then
\[
        |\calA|\le
        \max\Big\{\Psi(n,r_i,L_i)\prod_{j\ne i}\tbinom{n}{r_j} \colon i \in [m]\Big\}.
\]
\end{theorem}

\begin{proof}
For each $i$, let $X_i\coloneqq\tbinom{V_i}{r_i}$ and define a graph $G_i$ on $X_i$ by joining two distinct vertices $A,B\in X_i$ when $|A\cap B|\notin L_i$.  Then $\alpha(G_i)=\Psi(n,r_i,L_i)$.  It is easy to see that each $G_i$ is vertex-transitive.  In the categorical product $G_1\times\cdots\times G_m$, two tuples are adjacent exactly when their entries are adjacent in every coordinate.  Thus, if two tuples $(A_1,\ldots,A_m)$ and $(B_1,\ldots,B_m)$ of $\calA$ are adjacent in this product, then $|A_i\cap B_i|\notin L_i$ for every $i$, contradicting the defining property of $\calA$.  Hence $\calA$ is independent in the product.

We now spell out the use of Zhang's theorem.  For two vertex-transitive graphs $G,H$, Zhang's result~\cite{Zhang10} says that
\[
        \frac{\alpha(G\times H)}{|V(G)|\,|V(H)|}
        =
        \max\left\{
        \frac{\alpha(G)}{|V(G)|},
        \frac{\alpha(H)}{|V(H)|}
        \right\}.
\]
Let $P_k\coloneqq G_1\times\cdots\times G_k$.  The graph $P_k$ is vertex-transitive for every $k$, since the direct product of the automorphism groups of the factors acts transitively on $V(P_k)$.  Applying Zhang's two-factor theorem to $P_{k-1}$ and $G_k$ gives
\[
        \frac{\alpha(P_k)}{|V(P_k)|}
        =
        \max\left\{
        \frac{\alpha(P_{k-1})}{|V(P_{k-1})|},
        \frac{\alpha(G_k)}{|X_k|}
        \right\}.
\]
Induction on $k$ therefore gives
\[
        \frac{\alpha(G_1\times\cdots\times G_m)}{|X_1|\cdots |X_m|}
        =\max_{i\in[m]}\frac{\alpha(G_i)}{|X_i|}.
\]
Therefore
\[
        |\calA|\le\alpha(G_1\times\cdots\times G_m)
        =\max\Big\{\Psi(n,r_i,L_i)\prod_{j\ne i}\tbinom{n}{r_j} \colon i \in [m]\Big\}, 
\]
completing the proof of \cref{thm:product-intersection}. 
% If $r_i\in L_i$ for every $i$, choose a coordinate $i_0$ attaining the maximum, take an extremal $L_{i_0}$-intersecting family in coordinate $i_0$, and take the full product in the remaining coordinates.  Equality holds because tuples with the same coordinate in position $i_0$ have intersection size $r_{i_0}\in L_{i_0}$.
\end{proof}

Recall that $L_i=[0,d_i-1]\cup[b_i-c_i+d_i+1,b_i]$ for $i\in[m]$.  Let $F\coloneqq T^r_{a,b_1,c_1,d_1,\ldots,b_m,c_m,d_m}$ and put
\[
        N\coloneqq\prod_{i\in [m]}\tbinom{r}{b_i}
        \quad\text{and}\quad
        \Lambda\coloneqq
        \max\Big\{\Psi(r,b_i,L_i)\prod_{j\ne i}\tbinom{r}{b_j} \colon i \in [m]\Big\}.
\]

\begin{proposition}\label{prop:product-block}
Let $H$ be a finite $F$-hom-free $r$-graph with at least one edge, and let $(X_1,\ldots,X_r)$ be any symmetric random ordered edge of $H$ with ratio sequence $0<x_1\le\cdots\le x_r=1$.  Put $q_0\coloneqq a+\sum_{i\in [m]} b_i$ and $\beta\coloneqq\prod_{j\in [r]} x_j$.  Then $\beta \le \Lambda/N$. 
Consequently, by \cref{eq:hom-free-entropy-reduction},
\[
        \pi(F)\le\frac{\Lambda}{N}
        =\max_{i\in[m]}\frac{\Psi(r,b_i,L_i)}{\binom{r}{b_i}}.
\]
\end{proposition}

\begin{proof}
For each $i\in[m]$, fix an $r$-set $E_i=\{v_{i,1},\ldots,v_{i,r}\}$, and let $A=\{w_1,\ldots,w_a\}$ be disjoint from all $E_i$.  For every tuple
\[
        \mathbf B=(B_1,\ldots,B_m)\in\prod_{i\in [m]}\tbinom{[r]}{b_i},
\]
write $\hat{B}_i\coloneqq\left\{v_{i,j}:j\in B_i\right\}$, and let $P_{\mathbf B}$ be the maximal partial $r$-graph on $A\cup E_1\cup\cdots\cup E_m$ generated by $\{E_1,\ldots,E_m, S_{\mathbf B}\}$, where $S_{\mathbf B}\coloneqq A\cup\bigcup_{i\in [m]}\hat{B}_i$. 

Order the vertices by putting all vertices of $S_{\mathbf B}$ first and then, for each $i$, the vertices of $E_i\setminus\hat{B}_i$ in any fixed order.  This order makes $P_{\mathbf B}$ a partial forest.  Indeed, a vertex of $S_{\mathbf B}$ has as unique maximal face the initial segment of $S_{\mathbf B}$ ending at that vertex, while the $\ell$th vertex of $E_i\setminus\hat{B}_i$ has as unique maximal face $\hat{B}_i$ together with the first $\ell$ vertices of $E_i\setminus\hat{B}_i$.  Thus the vertices of $S_{\mathbf B}$ contribute one to each of $n_1,\ldots,n_{q_0}$, and the vertices of $E_i\setminus\hat{B}_i$ contribute one to each of $n_{b_i+1},\ldots,n_r$.  Hence \cref{lem:forest-sampling} gives a random homomorphism $Y^{\mathbf B}$ from $P_{\mathbf B}$ to $H$ with
\[
        \HH(Y^{\mathbf B})=(mr+a)\HH(X_1)+\log_2\Theta,
\]
where $\Theta\coloneqq \left(\prod_{j=r-q_0+1}^{r}x_j\right) \left(\prod_{i\in [m]}\prod_{j\in [r-b_i]}x_j\right)$. 

For $\phi\in V(H)^{A\cup E_1\cup\cdots\cup E_m}$, let $\calA(\phi)\coloneqq\left\{\mathbf B:\phi\in\supp(Y^{\mathbf B})\right\}$.  We claim that $\calA(\phi)$ satisfies the hypothesis of \cref{thm:product-intersection} with ground sets $E_i$ and parameters $b_i,L_i$.  If not, there are two distinct tuples $\mathbf B,\mathbf B'\in\calA(\phi)$, where $\mathbf B'=(B'_1,\ldots,B'_m)$, such that $|B_i\cap B'_i|\notin L_i$ for every $i$.  Thus
\[
        d_i\le |B_i\cap B'_i|\le b_i-c_i+d_i
        \quad\text{for}\quad i\in[m].
\]
Choose a $c_i$-set $D_i\subseteq B'_i$ with $|B_i\cap D_i|=d_i$.  This is possible because $|B_i\cap B'_i|\ge d_i$ and $|B'_i\setminus B_i|=b_i-|B_i\cap B'_i|\ge c_i-d_i$.  Put $\hat{D}_i\coloneqq\left\{v_{i,j}:j\in D_i\right\}$.  Then $P_{\mathbf B'}$ contains the face $A\cup\bigcup_i\hat{D}_i$, so $P_{\mathbf B}\cup P_{\mathbf B'}$ contains the partial hypergraph defining $F$, with partial edges $E_1,\ldots,E_m$, $A\cup\bigcup_i\hat{B}_i$, and $A\cup\bigcup_i\hat{D}_i$.  Since $\phi$ lies in both supports, these partial edges are mapped injectively into edges of $H$; by \cref{lem:extension-equivalence} this gives a homomorphism from $F$ to $H$, a contradiction.  Therefore $|\calA(\phi)|\le\Lambda$ for all $\phi$. 
\Cref{lem:mixture} gives a mixture $Z$ of the $Y^{\mathbf B}$ such that $N2^{\HH(Y^{\mathbf B})}\le\Lambda2^{\HH(Z)}$. 
The marginal of $Z$ on each $E_i$ has the law of $(X_1,\ldots,X_r)$, and each vertex of $A$ has the marginal law of $X_1$.  Hence
\[
        \HH(Z)\le m\HH(X_1,\ldots,X_r)+a\HH(X_1)
        =(mr+a)\HH(X_1)+m\log_2\beta.
\]
After substitution and cancellation, $N\Theta\le\Lambda\beta^m$. 
Each factor $x_j$ occurs in the first product defining $\Theta$ at most once and in the remaining $m$ products at most once for each $i$.  Hence its exponent in $\Theta$ is at most $m+1$.  Since $0<x_j\le1$, this implies $\Theta\ge\beta^{m+1}$.  Therefore $N\Theta\le\Lambda\beta^m$ gives $\beta\le\Lambda/N$.
By \cref{eq:hom-free-entropy-reduction}, this gives the asserted Tur\'an bound.  The identity for $\Lambda/N$ follows from the definitions of $\Lambda$ and $N$.
\end{proof}

%%%%%%%%%%%%%%%%%%%%%%%%%%%%%%%%%%%%%%%%
\section{Refinements via homomorphic reduction}\label{sec:hom-reductions}

This section gives refinements of \cref{thm:first-main,thm:product-main} obtained by applying homomorphic reductions to the extension constructions.  Each reduction below is realized by a direct quotient map.

\begin{lemma}\label{lem:first-hom-reductions}
The following statements hold.
\begin{enumerate}[label=\textup{(\roman*)},ref=\textup{(\roman*)}]
\item\label{itm:first-hom-reduction-one-edge} Assume \cref{eq:first-assumptions}.  Then $T^r_{a,b,c,d}\to K^r_r$ if and only if $a+b+c-d\le r$,  where $K^r_r$ denotes the \emph{one-edge $r$-graph}.  Consequently, $a+b+c-d\le r$ implies $\pi(T^r_{a,b,c,d})=0$.
\item\label{itm:first-hom-reduction-common-part} Assume \cref{eq:first-assumptions}. If $0\le s\le\min\{a,r-b-c+d\}$, then $T^r_{a,b,c,d}\longrightarrow T^r_{a-s,b+s,c+s,d+s}$. 
\item\label{itm:first-hom-reduction-first-side} Assume \cref{eq:first-assumptions}. If $0\le s\le\min\{r-a-b,c-d\}$, then $T^r_{a,b,c,d}\longrightarrow T^r_{a,b+s,c,d+s}$. 
\item\label{itm:first-hom-reduction-second-side} Assume \cref{eq:first-assumptions}. If $0\le s\le\min\{r-a-c,b-d\}$, then $T^r_{a,b,c,d}\longrightarrow T^r_{a,b,c+s,d+s}$. 
\item\label{itm:product-hom-reduction} Assume \cref{eq:product-coordinate-assumptions,eq:product-size-assumptions}.  Let $s_1,\ldots,s_m\ge0$ satisfy $s_i\le r-b_i-c_i+d_i$ for $i \in [m]$ and $\sum_{i\in [m]} s_i\le a$. Then $T^r_{a,b_1,c_1,d_1,\ldots,b_m,c_m,d_m} \to T^r_{a-\sum_i s_i, b_1+s_1,c_1+s_1,d_1+s_1, \ldots, b_m+s_m,c_m+s_m,d_m+s_m}$. 
\end{enumerate}
\end{lemma}

\begin{proof}
For \ref{itm:first-hom-reduction-one-edge}, if such a homomorphism exists, then $A\cup B\cup C$ must be mapped injectively, because any two of its vertices lie together in one of the three edges.  Thus $|A\cup B\cup C|=a+b+c-d\le r$.

Conversely assume $a+b+c-d\le r$.  Let $U_B$ and $U_C$ be the private vertex sets added to $A\cup B$ and $A\cup C$, respectively, in the $r$-uniform extension.  Put
\[
        I\coloneqq B\cap C,\qquad
        B_0\coloneqq B\setminus C,\qquad
        C_0\coloneqq C\setminus B,\quad\text{and}\quad
        R\coloneqq E\setminus(B\cup C),
\]
and let $t\coloneqq r-a-b-c+d\ge0$.  Choose decompositions
\[
        R=R_A\sqcup R_0,\qquad
        U_B=U_{B,C}\sqcup U_{B,0},\quad\text{and}\quad
        U_C=U_{C,B}\sqcup U_{C,0},
\]
where
\[
        |R_A|=a,\quad |R_0|=t,\quad
        |U_{B,C}|=c-d,\quad |U_{B,0}|=t,\quad
        |U_{C,B}|=b-d,\quad |U_{C,0}|=t.
\]
Color the vertices with $r$ colors as follows.  Give the vertices of $I$ singleton colors; pair the vertices of $B_0$ with those of $U_{C,B}$; pair the vertices of $C_0$ with those of $U_{B,C}$; pair the vertices of $A$ with those of $R_A$; and group the vertices of $R_0,U_{B,0},U_{C,0}$ into $t$ triples.  Each color class contains at most one vertex from each edge of the extension, and the number of color classes is
\[
        d+(b-d)+(c-d)+a+t=r.
\]
Mapping each color class to the corresponding vertex of $K^r_r$ gives a homomorphism.  The final assertion follows from \cref{prop:hom-monotonicity}.

For \ref{itm:first-hom-reduction-common-part}, use a presentation of $T^r_{a,b,c,d}$ with disjoint sets $A,E$ and subsets $B,C\subseteq E$ as in the definition.  Choose subsets $S\subseteq A$ and $D\subseteq E\setminus(B\cup C)$ with $|S|=|D|=s$, and choose a bijection $S\to D$.  Identify each vertex of $S$ with its paired vertex in $D$, while fixing all other vertices of $A\cup E$.

Put
\[
        A'\coloneqq A\setminus S,\qquad
        B'\coloneqq B\cup D,\qquad
        C'\coloneqq C\cup D.
\]
Then $|A'|=a-s$, $|B'|=b+s$, $|C'|=c+s$, and $|B'\cap C'|=d+s$.  The quotient sends the three partial edges $\{E, A\cup B, A\cup C\}$ injectively onto $\{E, A'\cup B', A'\cup C'\}$, because no identified pair is contained in any one of the source partial edges.
The private vertex sets added to $A\cup B$ and $A'\cup B'$ have the same size $r-a-b$, and similarly the private vertex sets added to $A\cup C$ and $A'\cup C'$ have the same size $r-a-c$.  Extending the quotient by arbitrary bijections between the corresponding private vertex sets gives the desired homomorphism of the $r$-uniform extensions.

For \ref{itm:first-hom-reduction-first-side}, let $U_B$ and $U_C$ be the private vertex sets added to $A\cup B$ and $A\cup C$, respectively, in the source extension.  Choose subsets $S\subseteq U_B$ and $D\subseteq C\setminus B$ with $|S|=|D|=s$, and identify the vertices of $S$ with the vertices of $D$ by an arbitrary bijection.  Put $B'\coloneqq B\cup D$ and $C'\coloneqq C$.  The image of the edge $E$ is injective, the image of $A\cup B\cup U_B$ is $A\cup B'\cup(U_B\setminus S)$, and the image of $A\cup C\cup U_C$ is $A\cup C'\cup U_C$.  These are the three edges in the extension of $T^r_{a,b+s,c,d+s}$, with $|B'\cap C'|=d+s$.  This gives the required homomorphism.

Part \ref{itm:first-hom-reduction-second-side} is symmetric.  Choose $S\subseteq U_C$ and $D\subseteq B\setminus C$ with $|S|=|D|=s$, identify them bijectively, and put $B'\coloneqq B$ and $C'\coloneqq C\cup D$.  The same edge-by-edge injectivity check gives a homomorphism to the extension of $T^r_{a,b,c+s,d+s}$.

For \ref{itm:product-hom-reduction}, choose pairwise disjoint subsets $A_i\subseteq A$ with $|A_i|=s_i$, which is possible because $\sum_i s_i\le a$.  For each $i$, choose a subset $D_i\subseteq E_i\setminus(B_i\cup C_i)$ with $|D_i|=s_i$, and identify the vertices of $A_i$ with the vertices of $D_i$ by an arbitrary bijection.  Put
\[
        A'\coloneqq A\setminus\bigcup_i A_i,\qquad
        B_i'\coloneqq B_i\cup D_i,\quad\text{and}\quad
        C_i'\coloneqq C_i\cup D_i.
\]
Then $|A'|=a-\sum_i s_i$, $|B_i'|=b_i+s_i$, $|C_i'|=c_i+s_i$, and $|B_i'\cap C_i'|=d_i+s_i$.  No identified pair lies in a common partial edge, so the quotient maps
\[
        E_1,\ldots,E_m,\quad
        A\cup\bigcup_i B_i,\quad\text{and}\quad
        A\cup\bigcup_i C_i
\]
injectively onto the corresponding partial edges of the target construction.  The two side partial edges have the same sizes before and after the quotient, so their private vertex sets also have the same sizes.  Extending over those private vertices by arbitrary bijections gives the claimed homomorphism.
\end{proof}

\begin{corollary}\label{cor:hom-improved-first}\label{cor:intro-hom-reduction}\label{cor:product-hom-reduction-bound}
The following statements hold.
\begin{enumerate}[label=\textup{(\roman*)},ref=\textup{(\roman*)}]
\item\label{itm:triangle-hom-reduction-bound} Assume \cref{eq:first-assumptions}.  Let 
\begin{align*}
    L_s^{(0)} & \coloneqq[0,d+s-1]\cup[b-c+d+s+1,b+s], \\
    L_s^{(1)} & \coloneqq[0,d+s-1]\cup[b-c+d+2s+1,b+s], \\
    L_s^{(2)} & \coloneqq[0,d+s-1]\cup[p_s-q_s+d+s+1,p_s], 
\end{align*}
where $p_s\coloneqq\max\{b,c+s\}$ and $q_s\coloneqq\min\{b,c+s\}$. 
Then
\[
    \begin{aligned}
        \pi(T^r_{a,b,c,d})
        \le \min\Bigg\{&
        \min_{0 \le s\le\min\{a,r-b-c+d\}}
        \frac{\Psi(r,b+s,L_s^{(0)})}{\tbinom{r}{b+s}},\\
        &
        \min_{0 \le s\le\min\{r-a-b,c-d\}}
        \frac{\Psi(r,b+s,L_s^{(1)})}{\tbinom{r}{b+s}},~
        \min_{0 \le s\le\min\{r-a-c,b-d\}}
        \frac{\Psi(r,p_s,L_s^{(2)})}{\tbinom{r}{p_s}}
        \Bigg\}.
    \end{aligned}
\]
\item\label{itm:product-hom-reduction-bound} Let $F\coloneqq T^r_{a,b_1,c_1,d_1,\ldots,b_m,c_m,d_m}$ with parameters satisfying \cref{eq:product-coordinate-assumptions,eq:product-size-assumptions}.  For every admissible vector $\mathbf s\coloneqq(s_1,\ldots,s_m)$ as in \cref{lem:first-hom-reductions}\ref{itm:product-hom-reduction}, put $L_i(s_i)\coloneqq[0,d_i+s_i-1]\cup[b_i-c_i+d_i+s_i+1,b_i+s_i]$. 
Then
\[
        \pi(F)
        \le
        \min_{\mathbf s}
        \max_{i\in[m]}
        \frac{\Psi(r,b_i+s_i,L_i(s_i))}{\tbinom{r}{b_i+s_i}}.
\]
\end{enumerate}
\end{corollary}

In the third minimum of \cref{cor:hom-improved-first}\ref{itm:triangle-hom-reduction-bound}, we use the elementary isomorphism
\[
        T^r_{a,b,c+s,d+s}\cong T^r_{a,p_s,q_s,d+s}
\]
obtained by interchanging the two side partial edges when $c+s>b$.

For instance, let $T_i^r\coloneqq T^r_{i,i,i,0}$, with $r\ge 2i$.  The homomorphic reduction $T^r_{a,b,c,d}\to T^r_{a,b+s,c,d+s}$ gives
\[
        \pi(T_i^r)
        \le
        \min_{0\le t\le\min\{i,r-2i\}}
        \frac{\Psi(r,i+t,[0,t-1]\cup[2t+1,i+t])}{\tbinom{r}{i+t}}.
\]
In particular, when $r=3i-1$ and $t=i-1$, any two distinct $(2i-1)$-subsets of $[3i-1]$ intersect in at least $i-1$ points, while the allowed intersection set is $[0,i-2]\cup\{2i-1\}$.  Hence $\Psi(3i-1,2i-1,[0,i-2]\cup\{2i-1\})=1$, and
\[
        \pi(T_i^{3i-1})
        \le
        \frac{1}{\tbinom{3i-1}{2i-1}}
        =
        \frac{1}{\tbinom{3i-1}{i}}.
\]
The first nontrivial instance is $\pi(T^5_{2,2,2,0})\le 1/10$.

%%%%%%%%%%%%%%%%%%%%%%%%%%%%%%%%%%%%%%%%
\section{Concluding remarks}

$\bullet$ In two previously studied cases, the upper bounds in this work recover the upper-bound side of known exact Tur\'an results.  First, when specialized to expanded triangles, \cref{thm:first-main} gives Frankl's upper bound; the corresponding lower bound is Frankl's theorem \cite{Frankl90}, giving the known equality for $T^{2k}_{k,k,k,0}=C^{2k}_3$. 
Second, for the $4$-uniform instance $T^4_{1,3,3,2}$, the $4$-graph on five vertices with three edges, \cref{thm:first-main} gives the upper bound $1/4$; this upper bound was already obtained by Gunderson--Semeraro via a modification of de Caen's counting argument \cite{deCaen83,GundersonSemeraro17}, and the corresponding lower bound follows from their construction \cite{GundersonSemeraro17}.  It remains natural to ask for which other extension constructions the upper bounds given by this work are exact.

$\bullet$  Beyond the $L$-intersection and matching problems used here, it would be natural to look for constructions controlled by other theorems from extremal set theory such as cross-intersecting families, shadows, forbidden configurations, packing and covering problems, or stability theorems.  Conversely, sharp and stability results for such set systems may suggest new candidate extremal constructions for hypergraph Tur\'an problems.

%%%%%%%%%%%%%%%%%%%%%%%%%%%%%%%%%
\section*{Acknowledgments}
Y.C. was supported by National Natural Science Foundation of China grant 123B2012.
X.L. was supported by the Excellent Young Talents Program (Overseas) of the National Natural Science Foundation of China.
T.Z. was supported by Innovation Program for Quantum Science and Technology 2021ZD0302902.

%%%%%%%%%%%%%%%%%%%%%%%%%%%%%%%%%
\bibliographystyle{abbrv}
\bibliography{EntropyTuran}
%%%%%%%%%%%%%%%%%%%%%%%%%%%%%%%%%
\end{document}